\documentclass[MSNbibl,number,citesort,seceqn,dvips]{arxbj}
\usepackage{upgreek}

\aid{0}
\volume{19}
\issue{2}
\pubyear{2013}
\firstpage{633}
\lastpage{645}
\doi{10.3150/12-BEJ417} 

\makeatletter
\newtheorem{them}{Theorem}[section]
\newtheorem{prop}[them]{Proposition}
\newtheorem{lem}[them]{Lemma}
\newtheorem{cor}[them]{Corollary}
\newremark{rems}{Step}
\newremark{rem}{Remark}
\newcommand{\xrightarrow}{\hbox to 11mm{\rightarrowfill}}

\renewcommand{\le}{\leq}
\renewcommand{\ge}{\geq}
\newcommand{\mcl}{\mathcal}
\newcommand{\mb}{\mathbf}
\newcommand{\mbx}{\mathbf{X}}
\newcommand{\E}{\mathbb{E}}
\newcommand{\N}{\mathbb{N}}
\newcommand{\R}{\mathbb{R}}
\newcommand{\Z}{\mathbb{Z}}
\renewcommand{\P}{\mathbb{P}}
\newcommand{\ov}{\overline}
\renewcommand{\d}{{\mathrm{d}}}
\renewcommand{\phi}{\varphi}
\makeatother

\begin{document}
\begin{frontmatter}

\title{On the rate of convergence in the martingale central limit theorem}
\runtitle{Rate of convergence in the martingale central limit theorem}

\begin{aug}
\author{\fnms{Jean-Christophe} \snm{Mourrat}\corref{}\ead[label=e1]{jean-christophe.mourrat@epfl.ch}}
\runauthor{J.-C. Mourrat} 
\address{Ecole polytechnique f\'{e}d\'{e}rale de Lausanne, institut
de math\'{e}matiques, station 8, 1015 Lausanne, Switzerland}
\end{aug}

\received{\smonth{6} \syear{2011}}
\revised{\smonth{10} \syear{2011}}

%
\begin{abstract}
Consider a discrete-time martingale, and let $V^2$ be its normalized
quadratic variation.
As $V^2$ approaches~$1$, and provided that some Lindeberg condition is
satisfied, the
distribution of the rescaled martingale approaches the Gaussian
distribution. For
any $p \ge1$, (\textit{Ann. Probab.} \textbf{16} (1988) 275--299)
gave a bound on the
rate of convergence in this central
limit theorem that is the sum of two terms, say $A_p + B_p$, where up
to a constant,
$A_p = {\|V^2-1\|}_p^{p/(2p+1)}$. Here we discuss the optimality of
this term, focusing
on the restricted class of martingales with bounded increments. In this context,
(\textit{Ann. Probab.} \textbf{10} (1982) 672--688) sketched a
strategy to prove optimality for $p = 1$. Here we extend this
strategy to any $p \ge1$, thereby justifying the optimality of the
term $A_p$. As a
necessary step, we also provide a new bound on the rate of convergence
in the central
limit theorem for martingales with bounded increments that improves on
the term $B_p$,
generalizing another result of (\textit{Ann. Probab.} \textbf{10}
(1982) 672--688).
\end{abstract}

%

\end{frontmatter}

\section{Introduction}\label{s:intro}
Let $\mb{X} = (X_1,\ldots,X_n)$ be a sequence of square-integrable
random variables such that for any~$i$, $X_i$ satisfies $\E[X_i \vert \mcl
{F}_{i-1}] = 0$, where $\mcl{F}_i$ is the $\sigma$-algebra generated
by $(X_1,\ldots,X_i)$. In other words, $\mbx$ is a square-integrable
martingale difference sequence. Following the notation of~\cite{bo},
we write $M_n$ for the set of all such sequences of length $n$, and introduce
\begin{eqnarray*}
s^2(\mb{X}) &=& \sum_{i = 1}^n \E[X_i^2],\\
V^2(\mb{X}) &=& s^{-2}(\mbx) \sum_{i = 1}^n \E[X_i^2 \vert \mcl{F}_{i-1}],\\
S(\mb{X}) &=& \sum_{i = 1}^n X_i.
\end{eqnarray*}
$V^2(\mb{X})$ can be called the normalized quadratic variation of $\mb{X}$.
Let $(\mb{X}_n)_{n \in\N}$ be such that for any $n$, $\mb{X}_n \in
M_n$. It is well known (see, e.g.,~\cite{durrett}, Section 7.7.a) that if
%
\begin{equation}\label{convprob}
V^2(\mb{X}_n) \mathop{\mathop{\xrightarrow}_{n \to+ \infty}}^{\mathrm{(prob.)}} 1
\end{equation}
and some Lindeberg condition is satisfied, then the rescaled sum $S(\mb
{X}_n)/s(\mb{X}_n)$ converges in distribution to a standard Gaussian
random variable, that is,
%
\begin{equation}\label{tcl}
\forall t \in\R,\qquad \P[S(\mb{X}_n)/s(\mb{X}_n) \le t ] \mathop{\xrightarrow}_{n \to+ \infty} \Phi(t),
\end{equation}
where $\Phi(t) = (2\uppi)^{-1/2} \int_{-\infty}^t \mathrm{e}^{-x^2/2} \,\d x$.

We are interested in bounds on the speed of convergence in this central
limit theorem. Several results have been obtained under a variety of
additional assumptions. One natural way to strengthen the convergence
in probability (\ref{convprob}) is to change it for a convergence in
$L^p$ for some $p \in[1,+\infty]$. Indeed, quantitative estimates in
terms of $\|V^2 - 1 \|_p$ seem particularly convenient when the aim is
to apply the result to practical situations. We write
\[
D(\mbx) = \sup_{t \in\R} | \P[S(\mb{X})/s(\mb{X}) \le t ] - \Phi
(t) |
\]
and
\[
\|\mb{X}\|_p = \max_{1 \le i \le n} \|X_i\|_p \qquad(p \in[1,+\infty]).
\]
\cite{ha88} proved the following result.
\begin{them}[(\cite{ha88})]
\label{t:ha}
Let $p \in[1 , +\infty)$. There exists a constant $C_{p} > 0$ such
that for any $n \ge1$ and any $\mbx\in M_n$,
%
\begin{equation}\label{eq:ha}
D(\mbx) \le C_{p} \Biggl({\|V^2(\mbx)-1\|}_p^{p} + s^{-2p}(\mbx) \sum_{i
= 1}^n {\|X_i\|}_{2p}^{2p} \Biggr)^{1/(2p+1)}.
\end{equation}
\end{them}

In~\cite{joos}, Theorem~\ref{t:ha} is generalized to the following.
\begin{them}[(\cite{joos})]
Let $p \in[1,+\infty]$ and $p' \in[1,+\infty)$. There exists
$C_{p,p'} > 0$ such that for any $n \ge1$ and any $\mbx\in M_n$,
%
\begin{equation}\label{eq:joos}
D(\mbx) \le C_{p,p'}\Biggl [{\|V^2(\mbx)-1\|}_p^{p/(2p+1)} + \Biggl(s^{-2p'}(\mbx
) \sum_{i = 1}^n {\|X_i\|}_{2p'}^{2p'}\Biggr )^{1/(2p'+1)} \Biggr].
\end{equation}
\end{them}

Here $p/(2p+1) = 1/2$ for $p = + \infty$. In fact, a stronger,
nonuniform bound is given; see~\cite{joos}, Theorem~2.2 (or,
equivalently,~\cite{pubjoos}), for details.

The main question addressed here concerns the optimality of the term
$\|V^2(\mbx)-\break 1\|_p^{p/(2p+1)}$ appearing in the right-hand side of
(\ref{eq:ha}) or (\ref{eq:joos}). About this,~\cite{ha88}
constructed a sequence of elements $\mbx_n \in M_n$ such that
\begin{itemize}
\item $s(\mbx_n) \simeq\sqrt{n}$,
\item $D(\mbx_n) \simeq\log^{-1/2}(n)$,
\item ${\|V^2(\mbx)-1\|}_p^{p} \simeq s^{-2p}(\mbx) {\|\mbx\|
}^{2p}_{2p} \simeq s^{-2p}(\mbx) \sum_{i = 1}^n {\|X_i\|}_{2p}^{2p}
\simeq\log^{-(2p+1)/2}(n)$,
\end{itemize}
where we write $a_n \simeq b_n$ if there exists $C > 0$, such that
$a_n/C \le b_n \le C a_n$ for all sufficiently large $n$. This example
demonstrates that the exponent $1/(2p+1)$ appearing on the outer
bracket of the right-hand side of (\ref{eq:ha}) cannot be improved. But
because the two terms of the right-hand side of (\ref{eq:ha}) are of
the same order, no conclusions can be drawn about the optimality of the
term $ {\|V^2(\mbx)-1\|}_p^{p/(2p+1)}$ alone. Most importantly, it is
rather disappointing that in the example, ${\|\mbx\|}^{2p}_{2p}$ and
$\sum_{i = 1}^n {\|X_i\|}_{2p}^{2p}$ are of the same order, if the
typical martingales that one is interested in have increments of
roughly the same order.

Using a similar construction, but also imposing the condition that
$V^2(\mbx) = 1$ a.s.,~\cite{joos}, Example~2.4, proved the optimality
of the exponent $1/(2p'+1)$ appearing in the second term of the sum in
the right-hand side of (\ref{eq:joos}). However, the author did not
discuss the optimality of the first term ${\|V^2(\mbx)-1\|}_p^{p/(2p+1)}$.

For $1 \le p \le2$, Theorem~\ref{t:ha} was in fact already proved by
\cite{hb}. In~\cite{hh}, Section~3.6, the authors could show
only that the bound on $D(\mbx)$ can be no better than $\|V^2(\mbx) -
1 \|_1^{1/2}$.

The proof of Theorem~\ref{t:ha} given by~\cite{ha88} is inspired by a
method introduced by~\cite{bo}, who proved the following results.
\begin{them}[(\cite{bo})]\label{t:bo}
Let $\gamma\in(0,+\infty)$. There exists a constant $C_{\gamma} >
0$ such that for any $n \ge2$ and any $\mbx\in M_n$ satisfying $\|
\mbx\|_{\infty} \le\gamma$ and $V^2(\mbx) = 1$ a.s.,
\[
D(\mbx) \le C_{\gamma} \frac{n \log(n)}{s^3(\mbx)}.
\]
\end{them}

Typically, $s(\mbx)$ is of order $\sqrt{n}$ when $\mbx\in M_n$.
Under such circumstances, Theorem~\ref{t:bo} thus gives a bound of
order $\log(n)/\sqrt{n}$. Moreover,~\cite{bo} provided an example of
a sequence of elements $\mbx_n \in M_n$ satisfying the conditions of
Theorem~\ref{t:bo}, such that $s^2(\mbx_n) = n$ and for which
\[
\limsup_{n \to+ \infty} \sqrt{n} \log^{-1}(n) D(\mbx_n) > 0,
\]
and so the result is optimal.

Relaxing the condition that $V^2(\mbx) = 1$ a.s.,~\cite{bo} then
showed the following result.
\begin{cor}[(\cite{bo})]
\label{cor:bo0}
Let $\gamma\in(0,+\infty)$. There exists a constant $\ov{C}_{\gamma
} > 0$ such that for any $n \ge2$ and any $\mbx\in M_n$ satisfying $\|
\mbx\|_{\infty} \le\gamma$,
%
\begin{equation}
\label{eq:bo0}
D(\mbx) \le\ov{C}_{\gamma} \biggl[ \frac{n \log(n)}{s^3(\mbx)} + \min
\bigl({\|V^2(\mbx)-1\|}_1^{1/3},{\|V^2(\mbx)-1\|}_\infty^{1/2} \bigr) \biggr].
\end{equation}
\end{cor}

See~\cite{joos}, Theorem~3.2, for a nonuniform version of this result.
A strategy was sketched by~\cite{bo} to prove that the bound $\|
V^2(\mbx) - 1 \|_1^{1/3}$ is indeed optimal, even on the restricted
class considered by Corollary~\ref{cor:bo0} of martingales with
bounded increments. This example provides a satisfactory answer to our
question of optimality for $p = 1$. The aim of the present paper is to
generalize Corollary~\ref{cor:bo0} and the optimality result to any $p
\in[1,+\infty)$. We begin by proving the following general result.
\begin{them}
\label{cor:bo}
Let $p \in[1,+\infty)$ and $\gamma\in(0,+\infty)$. There exists a
constant $C_{p,\gamma} > 0$ such that for any $n \ge2$ and any $\mbx
\in M_n$ satisfying $\|\mbx\|_{\infty} \le\gamma$,
%
\begin{equation}
\label{eq:bo}
D(\mbx) \le C_{p,\gamma} \biggl[ \frac{n \log(n)}{s^3(\mbx)} + \bigl({\|
V^2(\mbx)-1\|}_p^p + s^{-2p}(\mbx) \bigr)^{1/(2p+1)} \biggr].
\end{equation}
\end{them}

Note that, somewhat surprisingly, the term $s^{-2p}(\mbx) \sum_{i =
1}^n {\|X_i\|}_{2p}^{2p}$ appearing in inequality (\ref{eq:ha}) is no
longer present in (\ref{eq:bo0}), and is changed for the smaller
$s^{-2p}(\mbx)$ in (\ref{eq:bo}).

Finally, we justify the optimality of the term ${\|V^2(\mbx)-1\|
}_p^{p/(2p+1)}$ appearing in the right-hand side of (\ref{eq:bo}).
\begin{them}
\label{t:example}
Let $p \in[1,+\infty)$ and $\alpha\in(1/2,1)$. There exists a
sequence of elements $\mbx_n \in M_n$ such that
\begin{itemize}
\item $ \|\mbx_n\|_\infty\le2$,
\item $ s(\mbx_n) \simeq\sqrt{n}$,
\item $ {\|V^2(\mbx_n) - 1\|}_p^{p/(2p+1)} = \mathrm{O} ( n^{(\alpha- 1)/2} )$,
\item $ \limsup_{n \to+ \infty} n^{(1-\alpha)/2} D(\mbx_n) > 0$.
\end{itemize}
\end{them}

Our strategy for proving Theorem~\ref{t:example} builds on the
approach sketched by~\cite{bo} for the case where $p = 1$.
Interestingly, Theorem~\ref{cor:bo} is used in the proof of
Theorem~\ref{t:example}.

The question of optimality of the term ${\|V^2(\mbx)-1\|
}_p^{p/(2p+1)}$, now settled by Theorem~\ref{t:example}, arises
naturally in the problem of showing a quantitative central limit
theorem for the random walk among random conductances on $\Z^d$ \cite
{berry}. There, the random walk is approximated by a martingale. The
martingale increments are stationary and almost bounded for $d \ge3$,
in the sense that they have bounded $L^p$ norm for every $p < + \infty
$. Roughly speaking, for $d \ge3$, the variance of the rescaled
quadratic variation up to time $t$ decays as $t^{-1}$. This bound is
optimal and leads to a Berry--Esseen bound of order $t^{-1/5}$. Thus
Theorem~\ref{t:example} demonstrates that a better exponent of decay
than $1/5$ cannot be obtained when relying solely on information about
the variance of the quadratic variation.\looseness=-1

Theorem~\ref{cor:bo} is proved in Section~\ref{s:proofcor}, and
Theorem~\ref{t:example} is proved in Section~\ref{s:example}.
%
%
%
\section{\texorpdfstring{Proof of Theorem~\protect\ref{cor:bo}}{Proof of Theorem 1.5}}\label{s:proofcor}
The proof of Theorem~\ref{cor:bo} is essentially similar to the proof
of Corollary~\ref{cor:bo0} given by~\cite{bo}, with the additional
ingredient of a Burkholder inequality. Let $\mbx= (X_1,\ldots, X_n)
\in M_n$ be such that $\|\mbx\|_\infty\le\gamma$. The idea
(probably first suggested by~\cite{dvo}) is to augment\vadjust{\goodbreak} the sequence to
some $\hat{\mbx} \in M_{2n}$ such that $V^2(\hat{\mbx}) = 1$ a.s.,
while preserving the property that $\|\hat{\mbx}\|_\infty\le\gamma
$, and apply Theorem~\ref{t:bo} to this enlarged sequence. Let
\[
\tau= \sup \Biggl\{ k \le n \dvt \sum_{i=1}^k \E[X_i^2 \vert \mcl{F}_{i-1}] \le
s^2(\mbx) \Biggr\}.
\]
For $i \le\tau$, we define $\hat{X}_i = X_i$. Let $r$ be the largest
integer not exceeding
\[
\frac{s^2(\mbx) - \sum_{i = 1}^\tau\E[X_i^2 \vert \mcl
{F}_{i-1}]}{\gamma^2}.
\]
As $\|X\|_\infty\le\gamma$, clearly $r \le n$. Conditional on $\mcl
{F}_\tau$ and for $1 \le i \le r$, we let $\hat{X}_{i}$ be
independent random variables such that
$\P[\hat{X}_{\tau+ i} = \pm\gamma] = 1/2$. If $\tau+ r < 2n$,
then we let $\hat{X}_{\tau+ r + 1}$ be such that
\[
\P\Biggl[ \hat{X}_{\tau+ r + 1} = \pm \Biggl(s^2(\mbx) - \sum_{i = 1}^\tau\E
[X_i^2 \vert \mcl{F}_{i-1}] - r \gamma^2 \Biggr)^{1/2}\Biggr ] = \frac{1}{2},
\]
with the sign determined independent of everything else.
Finally, if $\tau+ r + 1 < 2n$, then we let $\hat{X}_{\tau+ r + i} =
0$ for $i \ge2$.

Possibly enlarging the $\sigma$-fields, we can assume that $\hat
{X}_i$ is $\mcl{F}_i$-measurable for $i \le n$, and define $\mcl
{F}_i$ to be the $\sigma$-field generated by $\mcl{F}_n$ and $\hat
{X}_{n+1},\ldots, \hat{X}_{n+i}$ if $i > n$. By construction, we have
\[
\sum_{i = \tau+ 1}^{2n} \E[\hat{X}_i^2 \vert \mcl{F}_{i-1}] = s^2(\mbx
) - \sum_{i = 1}^\tau\E[X_i^2 \vert \mcl{F}_{i-1}],
\]
which can be rewritten as
\[
\sum_{i = 1}^{2n} \E[\hat{X}_i^2 \vert \mcl{F}_{i-1}] = s^2(\mbx).
\]
Consequently, $s^2(\hat{\mbx}) = s^2(\mbx)$ and $V^2(\hat{\mbx}) =
1$ a.s. The sequence $\hat{\mbx}$ thus satisfies the assumptions of
Theorem~\ref{t:bo}, so
%
\begin{equation}
\label{Dhatx}
D(\hat{\mbx}) \le4 C_\gamma\frac{n \log(n)}{s^3(\mbx)}.
\end{equation}
For any $x > 0$, we have
%
\begin{eqnarray}\label{ineqxhatx}
\P\biggl[\frac{S(\mbx)}{s(\mbx)} \le t \biggr] & \le& \P\biggl[ \frac{S(\mbx
)}{s(\mbx)} \le t, \frac{|S(\mbx) - S(\hat{\mbx})|}{s(\mbx)} \le
x \biggr] + \P\biggl[ \frac{|S(\mbx) - S(\hat{\mbx})|}{s(\mbx)} \ge x \biggr]
\nonumber
\\[-8pt]
\\[-8pt]
& \le& \P\biggl[ \frac{S(\hat{\mbx})}{s(\mbx)} \le t+x\biggr ] + \frac
{1}{x^{2p}} \E \biggl[\biggl | \frac{S(\mbx) - S(\hat{\mbx})}{s(\mbx)} \biggr|^{2p} \biggr].\nonumber
\end{eqnarray}
Due to (\ref{Dhatx}), the first term in the right-hand side of (\ref
{ineqxhatx}) is smaller than
%
\begin{equation}
\label{boundPhi}
\Phi(t+x) + 4 C_\gamma\frac{n \log(n)}{s^3(\mbx)} \le\Phi(t) +
\frac{x}{\sqrt{2\uppi}} +4 C_\gamma\frac{n \log(n)}{s^3(\mbx)}.
\end{equation}
To control the second term, first note that
%
\begin{equation}
\label{SShat}
S(\mbx) - S(\hat{\mbx}) = \sum_{i = \tau+ 1}^{2n} (X_i - \hat{X}_i),
\end{equation}
where we put $X_i = 0$ for $i > n$. Given that $\tau+ 1$ is a stopping
time, conditional on $\tau$, the $(X_i - \hat{X}_i)_{i \ge\tau+ 2}$
still forms a martingale difference sequence. Thus we can use
Burkholder's inequality (see, e.g.,~\cite{hh}, Theorem 2.11), which
states that
%
\begin{eqnarray}
\label{multXXhat}
&&\frac{1}{C} \E\Biggl[ \Biggl| \sum_{i = \tau+ 2}^{2n} (X_i - \hat{X}_i) \Biggr|^{2p}
\Biggr] \nonumber
\\[-8pt]
\\[-8pt]
&&\quad\le\E \Biggl[ \Biggl( \sum_{i = \tau+ 2}^{2n} \E[(X_i - \hat{X}_i)^2 \vert \mcl
{F}_{i-1}]\Biggr )^p \Biggr] + \E \Bigl[ \max_{\tau+ 2 \le i \le2n} | X_i - \hat
{X}_i |^{2p} \Bigr],\nonumber
\end{eqnarray}
and we can safely discard the summand indexed by $\tau+1$ appearing in
(\ref{SShat}), which is uniformly bounded. The maximum on the
right-hand side of (\ref{multXXhat}) is also bounded by $2\gamma
^{2p}$. As for the other term, with $X_i$ and $\hat{X}_i$ as
orthogonal random variables, we have
%
\begin{eqnarray}
\label{eqnXXhat}
\sum_{i = \tau+ 1}^{2n} \E[(X_i - \hat{X}_i)^2 \vert \mcl{F}_{i-1}]
& = & \sum_{i = \tau+ 1}^{2n} \E[X_i^2 \vert \mcl{F}_{i-1}] + \sum_{i
= \tau+ 1}^{2n} \E[\hat{X}_i^2 \vert \mcl{F}_{i-1}] \nonumber
\\[-8pt]
\\[-8pt]
& = & s^2(\mbx) V^2(\mbx) + s^2(\mbx) - 2 \underbrace{\sum_{i =
1}^\tau\E[X_i^2 \vert \mcl{F}_{i-1}]}.\nonumber
\end{eqnarray}
Now, if $\tau= n$, then by definition the sum underbraced above is
$s^2(\mbx) V^2(\mbx)$. Otherwise, $\sum_{i = 1}^{\tau+1} \E[X_i^2
\vert \mcl{F}_{i-1}]$ exceeds $s^2(\mbx)$, but as the increments are
bounded, the sum underbraced is necessarily larger than $s^2(\mbx) -
\gamma^2$. In any case, we thus have
\[
\sum_{i = 1}^\tau\E[X_i^2 \vert \mcl{F}_{i-1}] \ge\min\bigl(s^2(\mbx
)V^2(\mbx), s^2(\mbx) - \gamma^2\bigr).
\]
Consequently, from (\ref{eqnXXhat}), we obtain that
\[
\sum_{i = \tau+ 1}^{2n} \E[(X_i - \hat{X}_i)^2 \vert \mcl{F}_{i-1}]
\le | s^2(\mbx) V^2(\mbx) - s^2(\mbx) | + 2 \gamma^2.
\]
Combining this with equations (\ref{multXXhat}), (\ref{SShat}), (\ref
{boundPhi}) and (\ref{ineqxhatx}), we finally obtain that
\[
\P\biggl[\frac{S(\mbx)}{s(\mbx)} \le t \biggr] - \Phi(t) \le4 C_\gamma\frac
{n \log(n)}{s^3(\mbx)} + \frac{x}{\sqrt{2\uppi}} + \frac{C}{x^{2p}}
\biggl( \|V^2(\mbx) - 1\|_p^p + \frac{\gamma^{2p}}{s^{2p}(\mbx)} \biggr).
\]
Optimizing this over $x > 0$ leads to the correct estimate. The lower
bound is obtained in the same way.
%
%
\section{\texorpdfstring{Proof of Theorem~\protect\ref{t:example}}{Proof of Theorem 1.6}}
\label{s:example}
Let $p \ge1$ and $\alpha\in(1/2,1)$ be fixed. We let $(X_{n i})_{1
\le i \le n - n^\alpha}$ be independent random variables with $\P
[X_{ni} = \pm1] = 1/2$. The subsequent $(X_{n i})_{n - n^\alpha< i
\le n}$ are defined recursively. Let
\[
\lambda_{n i} = \sqrt{n - i + \kappa_n^2},
\]
where $\kappa_n = n^{1/4}$ (in fact, any $n^\beta$ with $1-\alpha<
2\beta< \alpha$ would be fine).
Assuming that $X_{n,1},\ldots,X_{n,i-1}$ have been defined, we write
$\mathcal{F}_{n,i-1}$ for the $\sigma$-algebra that they generate,
and let
\[
S_{n, i-1} = \sum_{j = 1}^{i-1} X_{nj}.
\]
For any $i$ such that $n - n^\alpha< i \le n$, we construct $X_{n i}$
such that
%
\begin{equation}
\label{defXni}
\P[X_{n i} \in\cdot \vert \mathcal{F}_{n,i-1}] =
\left|
\begin{array}{l@{\qquad}l}
\delta_{-\sqrt{3/2}} + \delta_{\sqrt{3/2}} & \mbox{if } S_{n,i-1}
\in[\lambda_{n i}, 2 \lambda_{n i} ] , \\
\delta_{-\sqrt{1/2}}+ \delta_{\sqrt{1/2}} & \mbox{if } S_{n,i-1}
\in[-2 \lambda_{n i}, - \lambda_{n i} ], \\
\delta_{-1} + \delta_{1} & \mbox{otherwise},
\end{array} \right.
\end{equation}
where $\delta_x$ is the Dirac mass at point $x$. Here $(S_{ni})_{i \le
n}$ can be viewed as an inhomogeneous Markov chain. We write $\mbx_n =
(X_{n1},\ldots,X_{nn})$ and $\mbx_{ni} = (X_{n1},\ldots,X_{ni})$ for
any $i \le n$. Let
%
\begin{equation}
\label{defdeltan}
\delta(i) = \sup_{n \ge i} D(\mbx_{ni}).
\end{equation}

\begin{prop}
\label{p:deltaO}
Uniformly over $n$,
%
\begin{equation}
\label{V2}
\|V^2(\mbx_{ni}) - 1\|_p = \mathrm{O}\bigl ( i^{(\alpha- 1)(1+1/2p)}\bigr )\qquad (i \to+
\infty)
\end{equation}
and
%
\begin{equation}
\label{deltaO}
\delta(i) = \mathrm{O} \bigl(i^{(\alpha-1)/2}\bigr )\qquad (i \to+ \infty).
\end{equation}
\end{prop}

The proof goes as follows. First, we bound $\|V^2(\mbx_{ni}) - 1\|_p$
in terms of $(\delta(j))_{j \le i}$ in Lemma~\ref{l:deltaO}. This
gives an inequality on the sequence $(\delta(i))_{i \in\N}$ through
Theorem~\ref{cor:bo}, from which we deduce (\ref{deltaO}), and then
(\ref{V2}).
\begin{lem}
\label{l:deltaO}
Let $K_i = \max_{j \le i} \delta(j) j^{(1-\alpha)/2}$.
For any $n$ and $i$, the following inequalities hold:
%
\begin{eqnarray}
\label{square}
| \E[X_{ni}^2] - 1 | &\le&
\left|
\begin{array}{l@{\qquad}l}
0 & \mbox{if } i \le n - n^\alpha, \\
2 \delta(i-1) & \mbox{if } n - n^\alpha< i \le n,
\end{array}\right. \\
%
%
\label{s2}
| s^2(\mbx_{ni}) - i | &\le&
\left|
\begin{array}{l@{\qquad}l}
0 & \mbox{if } i \le n - n^\alpha, \\
C i^{(3\alpha- 1)/2} K_i \le C i^\alpha& \mbox{if } n - n^\alpha<
i \le n,
\end{array}\right.   \\
%
%
\label{V2delta}
\|V^2(\mbx_{ni}) - 1\|_p &\le&
\left|
\begin{array}{l@{\qquad}l}
0 & \mbox{if } i \le n - n^\alpha, \\
C i^{(\alpha- 1)(1+1/2p)} (1 + K_i)^{1/p} + C i^{(3\alpha- 3)/2} K_i
& \mbox{otherwise}.
\end{array} \right.    \qquad
\end{eqnarray}
\end{lem}
\begin{pf}
Inequality (\ref{square}) is obvious for $i \le n - n^\alpha$.
Otherwise, from the definition (\ref{defXni}), we know that
\[
\E[X_{ni}^2] = 1 + \tfrac{1}{2} \P[S_{n,i-1} \in I^+_{ni}] - \tfrac
{1}{2} \P[S_{n,i-1} \in I^-_{ni}],
\]
where we write
%
\begin{equation}
\label{defI}
I^+_{ni} = [\lambda_{ni},2\lambda_{ni}]\quad \mbox{and}\quad I^-_{ni} =
[-2\lambda_{ni}, -\lambda_{ni}].
\end{equation}
The random variable $S_{n,i-1}/s(\mbx_{n,i-1})$ is approximately
Gaussian, up to an error controlled by $\delta(i-1)$. More precisely,
\[
\biggl| \P[S_{n,i-1} \in I^+_{ni}] - \int_{I^+_{ni}/s(\mbx_{n,i-1})} \d
\Phi \biggr| \le2 \delta(i-1) .
\]
We obtain (\ref{square}) using the fact that
\[
\int_{I^+_{ni}/s(\mbx_{n,i-1})} \d\Phi= \int_{I^-_{ni}/s(\mbx
_{n,i-1})} \d\Phi.
\]
As a by-product, we also learn that
\[
| s^2(\mbx_{ni}) - i | \le
\left|
\begin{array}{l@{\qquad}l}
0 & \mbox{if } i \le n - n^\alpha, \\
2 \displaystyle\sum_{n - n^\alpha< j \le i} \delta(j-1) & \mbox{if } n -
n^\alpha< i \le n.
\end{array} \right.
\]
Recalling that $\alpha < 1$, we obtain (\ref{s2}), noting that for $n
- n^\alpha< i \le n$,
\[
\sum_{n - n^\alpha< j \le i} \delta(j-1) \le n^\alpha(n - n^\alpha
)^{(\alpha- 1)/2} K_i.
\]
In particular, it follows that
%
\begin{equation}
\label{s2isi}
s^2(\mbx_{ni}) = i\bigl(1+\mathrm{o}(1)\bigr).
\end{equation}
Turning now to (\ref{V2delta}), $\|V^2(\mbx_{ni}) - 1 \|_p$ is
clearly equal to $0$ for $i \le n - n^\alpha$, so let us assume the
contrary. We have:
%
\begin{eqnarray}
\label{eq1}
\|V^2(\mbx_{ni}) - 1 \|_p & = & s^{-2}(\mbx_{ni}) \Biggl\| \sum_{j = 1}^i
\E[X_{nj}^2 | \mcl{F}_{n,j-1}] - s^2(\mbx_{ni}) \Biggr\|_p \nonumber\\
& \le & \frac{1}{s^{2}(\mbx_{ni})} \sum_{j = 1}^i \| \E[X_{nj}^2 |
\mcl{F}_{n,j-1}] - 1 \|_p + \frac{ | s^2(\mbx_{ni}) - i
|}{s^{2}(\mbx_{ni}) } \\
& \le & \frac{1}{2 s^{2}(\mbx_{ni})} \sum_{n - n^\alpha< j \le i}
(\P[S_{n,j-1} \in I^+_{nj} \cup I^-_{nj}] )^{1/p} + \frac{ | s^2(\mbx
_{ni}) - i |}{s^{2}(\mbx_{ni}) }.\nonumber
\end{eqnarray}
We consider the two terms in (\ref{eq1}) separately. First, by the
definition of $\delta$, we know that
\[
\biggl| \P[S_{n,j-1} \in I^+_{nj} \cup I^-_{nj}] - \int_{(I^+_{nj} \cup
I^-_{nj})/s(\mbx_{n,j-1})} \d\Phi \biggr| \le2 \delta(j-1).
\]
Equation (\ref{s2isi}) implies that, uniformly over $j > n - n^\alpha$,
\[
\int_{(I^+_{nj} \cup I^-_{nj})/s(\mbx_{n,j-1})} \d\Phi= (2\uppi
)^{-1/2} \frac{2 \lambda_{nj}}{s(\mbx_{n,j-1})}\bigl (1 + \mathrm{o}(1)\bigr) \le C
n^{(\alpha- 1)/2},
\]
and so the first term of (\ref{eq1}) is bounded by
%
%
\begin{eqnarray}\label{bound1}
&& \frac{C}{i} \sum_{n - n^\alpha< j \le i} \bigl(n^{(\alpha- 1)/2} + 2
\delta(j-1)\bigr)^{1/p} \nonumber\\
&&\quad\le \frac{C}{i} \sum_{n - n^\alpha< j \le i} \bigl(n^{(\alpha- 1)/2}
+ 2 (n - n^\alpha)^{(\alpha- 1)/2} K_i\bigr)^{1/p} \\
&&\quad \le C i^{(\alpha- 1)(1+1/2p)} (1 + K_i)^{1/p}.\nonumber
\end{eqnarray}
The second term in (\ref{eq1}) is controlled by (\ref{s2}), and we
obtain inequality (\ref{V2delta}).
\end{pf}
\begin{pf*}{Proof of Proposition~\ref{p:deltaO}}
Applying Theorem~\ref{cor:bo} with the information given by Lemma~\ref
{l:deltaO}, we obtain that, up to a multiplicative constant that does
not depend on $n$ and $i \le n$, $D(\mbx_{ni})$ is bounded by
%
\begin{equation}
\label{eq2}
\frac{\log(i)}{\sqrt{i}} + i^{(\alpha- 1)/2} (1 + K_i)^{1/(2p+1)} +
i^{- 3(1 - \alpha)p/(4p+2)} K_i^{p/(2p+1)} + i^{-p/(2p+1)}.
\end{equation}
The first term can be disregarded, because it is dominated by
$i^{-p/(2p+1)}$. Also note that as $p \ge1$, we have
\[
\frac{3(1 - \alpha)p}{4p+2} \ge \frac{1-\alpha}{2},
\]
and as $\alpha> 1/2 > 1/(2p+1)$, we also have
\[
\frac{p}{2p+1} \ge\frac{1-\alpha}{2}.
\]
Multiplying (\ref{eq2}) by $i^{(1-\alpha)/2}$, we thus obtain
\[
K_i \le C (1 + K_i)^{1/(2p+1)} + C K_i^{p/(2p+1)},
\]
where we recall that the constant $C$ does not depend on $i$. Observing
that the set $\{x \ge0 \dvt x \le C (1+x)^{1/(2p+1)} + C x^{p/(2p+1)} \}$
is bounded, we obtain that $K_i$ is a bounded sequence, so (\ref
{deltaO}) is proved. The relation (\ref{V2}) then follows from (\ref
{deltaO}) and (\ref{V2delta}).
\end{pf*}
\begin{prop}
\label{p:ex2}
We have
\[
\limsup_{i \to+ \infty} i^{(1-\alpha)/2} \delta(i) > 0.
\]
\end{prop}
\begin{pf}
Our aim is to contradict, by reductio ad absurdum, the claim that
%
\begin{equation}
\label{cond*}
\delta(i) = \mathrm{o}\bigl (i^{(\alpha-1)/2} \bigr) \qquad(i \to+ \infty).
\end{equation}
Let $Z_1,\ldots,Z_n$ be independent standard Gaussian random
variables, and let $\xi_n$ be an independent centered Gaussian random
variable with variance $\kappa_n^2$, all independent of $\mbx_{n}$.
Assuming (\ref{cond*}), we contradict the fact that
%
\begin{equation}
\label{contrad}
D(\mbx_n) = \mathrm{o} \bigl(n^{(\alpha- 1)/2}\bigr ).
\end{equation}
Let $W_{ni} = \sum_{j = {i+1}}^n Z_j + \xi_n$. Noting that $n^{-1/2}
\sum_{j = 1}^n Z_j$ is a standard Gaussian random variable, and with
the aid of~\cite{bo}, Lemma~1, we learn that
\[
\biggl| \P[W_{n0} \le0] - \frac{1}{2}\biggr | \le C \frac{\kappa_n}{\sqrt{n}}
\]
and, similarly,
\[
\biggl| \P[S_{nn} + \xi_n \le0] - \frac{1}{2} \biggr| \le C \biggl(D(\mbx_n) + \frac
{\kappa_n}{s(\mbx_{n})} \biggr).
\]
Combining these two observations with (\ref{s2}), we thus obtain that
%
\begin{equation}
\label{decompSn}
\P[S_{nn} + \xi_n \le0] - \P[W_{n0} \le0] \le C \biggl( D(\mbx_n) +
\frac{\kappa_n}{\sqrt{n}} \biggr).
\end{equation}
As $\kappa_n = n^{1/4}$ and $\alpha> 1/2$, we know that $\kappa
_n/\sqrt{n} = \mathrm{o}(n^{(\alpha- 1)/2})$. We decompose the left-hand side
of (\ref{decompSn}) as
\[
\sum_{i = 1}^n \P[S_{n,i-1} + X_{ni} + W_{ni} \le0] - \P[S_{n,i-1}
+ Z_i + W_{ni} \le0].
\]
The random variable $W_{ni}$ is Gaussian with variance ${\lambda
}_{ni}^2 = n-i+\kappa_n^2$ and is independent of $\mbx_n$; thus the
sum can be rewritten as
%
\begin{equation}
\label{sum}
\sum_{i = 1}^n \E\biggl[ \Phi\biggl(-\frac{S_{n,i-1} + X_{ni}}{{\lambda}_{ni}}
\biggr) - \Phi\biggl(-\frac{S_{n,i-1} + Z_i}{{\lambda}_{ni}} \biggr) \biggr].
\end{equation}
Let $\phi(x) = (2\uppi)^{-1/2} \mathrm{e}^{-x^2/2}$. We can replace
%
\begin{equation}
\label{eq3}
\Phi\biggl(-\frac{S_{n,i-1} + X_{ni}}{{\lambda}_{ni}} \biggr)
\end{equation}
by its Taylor expansion,
%
\begin{equation}
\label{taylor}
\Phi\biggl(-\frac{S_{n,i-1}}{{\lambda}_{ni}} \biggr) - \frac{X_{ni}}{{\lambda
}_{ni}} \phi\biggl( -\frac{S_{n,i-1}}{{\lambda}_{ni}} \biggr) + \frac
{X_{ni}^2}{2 {\lambda}_{ni}^2} \phi' \biggl( -\frac{S_{n,i-1}}{{\lambda
}_{ni}} \biggr),
\end{equation}
up to an error bounded by
%
\begin{equation}
\label{Taylorerror}
\frac{|X_{ni}|^3}{6{\lambda}_{ni}^3} \|\phi''\|_{\infty}.
\end{equation}
\begin{rems}\label{step1}
We show that the error term (\ref{Taylorerror}), after
integration and summation over $i$, is $\mathrm{o}(n^{(\alpha- 1)/2})$. Because
$X_{ni}$ is uniformly bounded, it suffices to show that
%
\begin{equation}
\label{sumo1}
\sum_{i = 1}^n \frac{1}{\lambda_{ni}^3} = \mathrm{o}\bigl (n^{(\alpha- 1)/2} \bigr).
\end{equation}
The foregoing sum equals
\[
\sum_{i = 1}^n \frac{1}{(n-i+\kappa_n^2)^{3/2}} \le n^{-1/2} \int
_{(\kappa_n^2 - 1)/n}^{(\kappa_n^2 +n)/n} x^{-3/2} \,\d x = \mathrm{O} (\kappa
_n^{-1} ).
\]
Because we defined $\kappa_n$ to be $n^{1/4}$ and $\alpha> 1/2$,
equation (\ref{sumo1}) is proved.
\end{rems}
\begin{rems}\label{step2}
For the second part of the summands in (\ref{sum}),
the same holds with $X_{ni}$ replaced by $Z_i$ and, similarly,
%
\begin{equation}
\label{sumo2}
\sum_{i = 1}^n \frac{\E[|Z_i|^3]}{{\lambda}_{ni}^3} = \mathrm{o}\bigl (n^{(\alpha
- 1)/2}\bigr ).
\end{equation}
\end{rems}
\begin{rems}
Combining the results of the two previous steps, we
know that up to a term of order $\mathrm{o} (n^{(\alpha- 1)/2} )$, the sum in
(\ref{sum}) can be replaced by
\[
\sum_{i = 1}^n \E\biggl[ \frac{Z_i - X_{ni}}{{\lambda}_{ni}} \phi\biggl(
-\frac{S_{n,i-1}}{{\lambda}_{ni}}\biggr ) + \frac{X_{ni}^2 - Z_i^2}{2
{\lambda}_{ni}^2} \phi' \biggl( -\frac{S_{n,i-1}}{{\lambda}_{ni}}\biggr ) \biggr].
\]
Conditional on $S_{n,i-1}$, both $Z_i$ and $X_{ni}$ are centered random
variables; thus the first part of the summands vanishes, and only the
following remains:
%
\begin{equation}
\label{sumcondi}
\sum_{i = 1}^n \E\biggl[ \frac{X_{ni}^2 - Z_i^2}{2 {\lambda}_{ni}^2} \phi
' \biggl( -\frac{S_{n,i-1}}{{\lambda}_{ni}} \biggr) \biggr] = \sum_{i = 1}^n \E\biggl[
\frac{\E[{X_{ni}^2 - 1 \vert S_{n,i-1}]}}{2 {\lambda}_{ni}^2} \phi' \biggl(
-\frac{S_{n,i-1}}{{\lambda}_{ni}} \biggr) \biggr].
\end{equation}
From the definition of $X_{ni}$, we learn that $\E[X_{ni}^2 - 1 \vert
S_{n,i-1}]$ is $0$ if $i \le n - n^\alpha$ but otherwise equals
\[
\left|
\begin{array}{l@{\qquad}l}
1/2 & \mbox{if } S_{n,i-1} \in I^+_{ni}, \\
- 1/2 & \mbox{if } S_{n,i-1} \in I^-_{ni}, \\
0 & \mbox{otherwise},
\end{array}   \right.
\]
where $I^+_{ni}$ and $I^-_{ni}$ are as defined in (\ref{defI}).
Consequently, it is clear that the contribution of each summand in the
right-hand side of (\ref{sumcondi}) is positive. Moreover, for $i > n
- n^\alpha$ and in the case where $S_{n,i-1} \in I^-_{ni} \cup
I^+_{ni}$, we have
\[
\E[X_{ni}^2 - 1 \vert S_{n,i-1}] \phi' \biggl( -\frac{S_{n,i-1}}{{\lambda
}_{ni}} \biggr) \ge\frac{1}{2} \inf_{[1,2]} |\phi' | > 0.
\]
Let us assume temporarily that, uniformly over $n$ and $i$ such that $n
- n^\alpha< i \le n - (n^\alpha)/2$, we have
%
\begin{equation}
\label{tempassump}
\P[S_{n,i-1} \in I^-_{ni} \cup I^+_{ni}] \ge C \frac{\lambda
_{ni}}{\sqrt{n}}.
\end{equation}
Then the sum in the right-hand side of (\ref{sumcondi}) is, up to a
constant, bounded from below by
\[
\sum_{n - n^\alpha< i \le n - (n^\alpha)/2} \frac{1}{\lambda_{ni}
\sqrt{n}} \ge C {n^\alpha} \frac{1}{n^{\alpha/2} \sqrt{n}} = C
{n^{(\alpha-1)/2}}.
\]
This contradicts (\ref{contrad}) via inequality (\ref{decompSn}), and
thus completes the proof of the proposition.
\end{rems}
\begin{rems}
There remains to show (\ref{tempassump}), for $n -
n^\alpha< i \le n - (n^\alpha)/2$. We have
\[
\biggl|\P[S_{n,i-1} \in I^+_{ni}] - \int_{I^+_{ni}/s(\mbx_{n,i-1})} \d
\Phi \biggr| \le2 \delta(i-1).
\]
Using inequality (\ref{s2}), it follows that
\[
\int_{I^+_{ni}/s(\mbx_{n,i-1})} \d\Phi\ge C \frac{\lambda
_{ni}}{\sqrt{n}}.
\]
Because we choose $i$ inside $[n - n^\alpha, n - (n^\alpha)/2]$,
$\lambda_{ni}$ is larger than $Cn^{\alpha/2}$, whereas $\delta(i -
1) = \mathrm{o} (i^{(\alpha- 1)/2} )$ by assumption (\ref{cond*}). This proves
(\ref{tempassump}).
\end{rems}
\upqed\end{pf}
\begin{rem*} To match the example proposed by~\cite{bo}, $\alpha
= 1/3$ and $\kappa_n = 1$ should be used in the definition of the
sequences $(\mbx_n)$. In this case, Propositions~\ref{p:deltaO} and
\ref{p:ex2} still hold. Although the proof of Proposition~\ref
{p:deltaO} can be kept unchanged, Proposition~\ref{p:ex2} requires a
more subtle analysis. First, $\xi_n$ of variance $\ov{\kappa}_n^2
\neq1$ must be chosen, which requires changing the $\lambda_{ni}$
appearing in (\ref{sum}) by, say, $\ov{\lambda}_{ni} = \sqrt{n - i
+ \ov{\kappa}_n^2}$. The sequence $\ov{\kappa}_n^2$ should grow to
infinity with $n$, while remaining $\mathrm{o}(n^\alpha)$. In Step~\ref{step1}, bounding
the difference between (\ref{eq3}) and (\ref{taylor}) by (\ref
{Taylorerror}) is too crude. Instead, it can be bounded by
\[
\frac{C}{\ov{\lambda}_{ni}^3} \Psi\biggl( - \frac{S_{n,i-1}}{\ov
{\lambda}_{ni}} \biggr),
\]
where $\Psi(x) = \sup_{|y| \le1} |\phi''(x+y)|$. One can then
appeal to~\cite{bo}, Lemma 2, and get through this step, using the fact
that $\ov{\kappa}_n$ tends to infinity. Step~\ref{step2} is similar, but with
some additional care required because $Z_i$ is unbounded. The rest of
the proof then applies, taking note of the discrepancy between $\lambda
_{ni}$ and $\ov{\lambda}_{ni}$ when necessary.

\end{rem*}
\section*{Acknowledgements}
I would like to thank an anonymous referee
for his careful review, and for his mention of the reference \cite
{joos}, which I was not aware of.


\printhistory

\end{document}